\documentclass[a4paper,reqno,10pt]{amsart}
\usepackage{amsfonts,amssymb,amsthm,amsbsy,latexsym,amscd,amsmath,euscript,enumerate,manfnt, marvosym,verbatim,calc,mathrsfs,marvosym,tensor,textcomp ,color,xcolor,tikz,setspace,upgreek,bbm,bm}
\usepackage[linkcolor=blue,colorlinks=true]{hyperref}

\newtheorem{theorem}{Theorem}[section]
\newtheorem{corollary}[theorem]{Corollary}

\theoremstyle{definition}
\newtheorem{definition}[theorem]{Definition}
\newtheorem{construction}[theorem]{Construction}
\newtheorem*{conjecture}{Conjecture}

\newtheorem*{acknowledgement}{Acknowledgement}

\newcommand{\tr}[2][]{\ensuremath{{\textnormal{tr}}_{#1}(#2)}}
\newcommand{\vs}[2][]{\ensuremath{{\textnormal{v}}_{#1}(#2)}}
\newcommand{\uk}[2][]{\ensuremath{{\textnormal{k}}_{#1}(#2)}}

\makeatletter
\def\imod#1{\allowbreak\mkern10mu({\operator@font mod}\,\,#1)}
\def\@textbottom{\vskip\z@\@plus 18pt}
\let\@texttop\relax
\makeatother

\title[Points in a maximal intersecting family of $k-$sets]{\textnormal{On the maximum number of points\\ in a maximal intersecting family of finite sets}}
\author{Kaushik Majumder}
\date{February~28, 2014}

\address{\newline Theoretical Statistics and Mathematics Unit\\ \newline Indian Statistical Institute, Bangalore Centre \\ \newline $8^{th}$ Mile\\ Mysore Road, Bangalore - $560059$, India.\newline \textnormal{\textestimated-Mail}: {\tt kaushik\MVAt isibang.ac.in}}
 

\begin{document}

\begin{abstract}
Paul Erd\H{o}s and L\'{a}szl\'{o} Lov\'{a}sz proved in a landmark article that, for any positive integer $k$, up to isomorphism there are only finitely many \emph{maximal intersecting families of $k-$sets} (maximal $k-$cliques). So they posed the problem of determining or estimating the largest number $N(k)$ of the points in such a family. They also proved by means of an example that $N(k)\geq2k-2+\frac{1}{2}\binom{2k-2}{k-1}$. Much later, Zsolt Tuza proved that the bound is best possible up to a multiplicative constant by showing that asymptotically $N(k)$ is at most $4$ times this lower bound. In this paper we reduce the gap between the lower and upper bound by showing that asymptotically $N(k)$ is at most $3$ times the Erd\H{o}s-Lov\'{a}sz lower bound. Conjecturally, the explicit upper bound obtained in this paper is only double the lower bound.  
\end{abstract}
\maketitle

\section{Introduction}

By a \emph{family} we shall mean a family of finite sets. For a family $\mathcal{F}$, the members of $\mathcal{F}$ are called its \emph{blocks} and the elements of the blocks are called its \emph{points}. In other words, the \emph{point set} of $\mathcal{F}$, denoted by $P_{\mathcal{F}}$, is the union of all its blocks. In case $\mathcal{F}$ is finite, we shall denote its number of points (size of the point set) by $\vs{\mathcal{F}}$.

A family $\mathcal{F}$ is said to be \emph{uniform} if all its blocks have the same size. If $\mathcal{F}$ is a uniform family we shall denote its common block size by $\uk{\mathcal{F}}$. A \emph{blocking set} of a family $\mathcal{F}$ is a set $A$ which intersects every block of $\mathcal{F}$. We define a \emph{transversal} of $\mathcal{F}$ to be a blocking set of $\mathcal{F}$ with the smallest possible size \--- in case $\mathcal{F}$ has a finite blocking set. In this case we denote by $\tr{\mathcal{F}}$ the common size of its transversals. If $\mathcal{F}$ has no finite blocking set we may put $\tr{\mathcal{F}}=\infty$. (Warning: Many authors use the word transversal as a synonym for blocking sets.) If $\tr{\mathcal{F}}<\infty$, we denote the family of transversals of $\mathcal{F}$ by $\mathcal{F}^{\top}$. Note that $\mathcal{F}^{\top}$ is a uniform family with $\uk{\mathcal{F}^{\top}}=\tr{\mathcal{F}}$. Now we introduce:-
\begin{definition}
A family $\mathcal{F}$ is said to be a \emph{maximal intersecting family} (in short $MIF$) if $\tr{\mathcal{F}}<\infty$ and $\mathcal{F}=\mathcal{F}^{\top}$. We use $MIF(k)$ as a generic name for $MIF$'s with $\uk{\mathcal{F}}=k$.
\end{definition}

We say that a family $\mathcal{F}$ is an \emph{intersecting family} if any two blocks of $\mathcal{F}$ have non empty intersection. Clearly any $MIF(k)$ is an intersecting family. Indeed, the $MIF(k)$'s are characterized among all $k-$uniform intersecting families as those families which are maximal with respect to the property of being intersecting. Thus, an intersecting family $\mathcal{F}$ of $k-$sets is a $MIF(k)$ if and only if there is no $k-$set outside $\mathcal{F}$ (anywhere in the universe of all sets!) which is a blocking set of $\mathcal{F}$. In the hypergraph literature these are known as the \emph{maximal $k-$cliques}.

In \cite{MR0382050} Erd\H{o}s and Lov\'{a}sz proved the surprising result that any $MIF(k)$ is finite; indeed it has at most $k^{k}$ blocks. In Theorem~\ref{k^t} we point out that, more generally, for any $k-$uniform family $\mathcal{F}$ with finite transversal size $\tr{\mathcal{F}}=t$, the family $\mathcal{F}^{\top}$ is finite. Indeed, $\#(\mathcal{F}^{\top})\leq k^{t}$.

In view of the result of Erd\H{o}s and Lov\'{a}sz quoted above, we see that, for any fixed $k\geq1$, there are only finitely many $MIF(k)$'s, up to isomorphism. This led Erd\H{o}s and Lov\'{a}sz to ask for the determination of the maximum possible number $N(k)$ of points among all $MIF(k)$'s. Thus
\begin{equation*}
N(k):=\max\left\{\vs{\mathcal{F}}: \mathcal{F} \textup{ is a }MIF(k)\right\}.
\end{equation*}
By means of an explicit construction in \cite{MR0382050} it was proved that
\begin{equation}\label{lower_bound}
N(k)\geq 2k-2+\frac{1}{2}\binom{2k-2}{k-1}.
\end{equation}
Note that the lower bound in \eqref{lower_bound} is asymptotically $\frac{1}{2}\binom{2k-2}{k-1}$. In 1985. Tuza proved that, up to a multiplicative constant this is best possible. In order to explain Tuza's contribution, we recall
\begin{definition}
An \emph{intersecting set pair system} (in short $ISP$) is a collection $\{(A_{i}, B_{i}):1\leq i\leq n\}$ of pairs of finite sets with the property that, for $1\leq i,j\leq n$, $A_{i}\cap B_{j}=\emptyset$ if and only if $i=j$. Clearly, in such a system, the sets $A_{i}$ (as well as the sets $B_{i}$) are distinct. The set $\overset{n}{\underset{i=1}\cup}(A_{i}\sqcup B_{i})$ is called the \emph{point set} of the $ISP$. We denote by $\vs{\mathbb{I}}$ the number of points of an $ISP$ $\mathbb{I}$. If in $\mathbb{I}$, $\#(A_{i})=k$ and $\#(B_{i})=t$ for $1\leq i\leq n$, then we say that $\mathbb{I}$ is an $ISP$ with parameter $(k,t)$. We use $ISP(k,t)$ as a generic name for an $ISP$ with parameter $(k,t)$.
\end{definition}
In \cite{MR0183653}, Bollob\'{a}s proved the following inequality for arbitrary $ISP$'s. If $\{(A_{i}, B_{i}):1\leq i\leq n\}$ is an $ISP$ then 
\begin{equation}\label{Bollobas_inequality1}
\overset{n}{\underset{i=1}\sum}\frac{1}{\binom{\#(A_{i})+\#(B_{i})}{\#(A_{i})}}\leq 1. 
\end{equation}
In particular, for any $ISP(k,t)$ consisting of $n$ pairs, we have Bollob\'{a}s inequality
\begin{equation}\label{Bollobas_inequality2}
n\leq\binom{k+t}{k}.
\end{equation}
This inequality shows that, for any two positive integers $k$ and $t$, there are only finitely many $ISP(k,t)$, up to isomorphism. This raises the question of determining or estimating the number
\begin{equation*}
n(k,t):=\max\left\{\vs{\mathbb{I}}: \mathbb{I}\textup{ is an }ISP(k,t)\right\}
\end{equation*}
Notice that we have $n(k,t)=n(t,k)$. 

In Theorem~6(a) of \cite{MR811117}, Tuza used an extremely elegant argument to deduce the following bound from  Bollob\'{a}s inequality~\eqref{Bollobas_inequality1}. (The sum here is a simplification of the sum given by Tuza.)
\begin{equation}\label{Tuza_upperbound_t}
\textup{ For }k\geq t,\hspace{4mm} n(k,t)\leq\binom{k+t}{t+1}-\binom{2t-1}{t+1}+\frac{3}{2}\overset{t-1}{\underset{i=1}\sum}\binom{2i}{i}.
\end{equation}

A family $\mathcal{F}$ is $1-$critical according to Tuza if for any $x\in B\in\mathcal{F}$, there is a $B^{'}\in\mathcal{F}$ such that $B\cap B^{'}=\{x\}$. Notice that any $MIF(k)$ is $1-$critical (else $B\smallsetminus\{x\}$ would be a blocking set). In Corollary~12 of \cite{MR811117}, Tuza observes that $n(k,k-1)$ is an upper bound on the number of points in any $k-$uniform $1-$critical family. In particular this applies to $MIF(k)$'s. So we have 
\begin{equation}\label{N_Tuza}
N(k)\leq n(k,k-1).
\end{equation}
Substituting $t=k-1$ in \eqref{Tuza_upperbound_t} we therefore get
\begin{align}
N(k)\leq&\frac{3}{2}\overset{k-1}{\underset{i=1}\sum}\binom{2i}{i}\sim2\binom{2k-2}{k-1}.\label{asymptotic_Tuza}
\end{align}
where the asymptotics is determined by Stirling's asymptotic formula for factorials and summation by parts. Thus, as $k\rightarrow\infty$, Tuza's upper bound is $4$ times the lower bound given by Erd\H{o}s and Lov\'{a}sz.

The main object of this paper is to improve the estimate \eqref{asymptotic_Tuza} on $N(k)$. The method adopted here is inspired by Tuza \cite{MR811117}. We introduce the problem of finding or estimating the number
\begin{equation*}
N^{\top}(k,t):=\max\left\{\vs{\mathcal{F}^{\top}}:\mathcal{F}\textup{ is a uniform family with }\uk{\mathcal{F}}=k\textup{ and }\tr{\mathcal{F}}=t\right\}.
\end{equation*}
(Note that we are trying to maximize the size of the point set of the family of transversals of $\mathcal{F}$, which in general is a subset of the point set of $\mathcal{F}$.) This number is finite in view of Theorem~\ref{k^t} below. In Theorem~\ref{NT(k,t)} we prove:
\begin{equation}
N^{\top}(k,t)\leq n(k,t-1).
\end{equation}
In Theorem~\ref{x_mapsto_y}, we show that, given any $MIF(k)$ $\mathcal{F}$ such that $\mathcal{F}$ has a pair $\{\alpha,\beta\}$ of points not contained in any block of $\mathcal{F}$, one can construct another $MIF(k)$, denoted $\mathcal{F}[\beta\mapsto\alpha]$, with one less point. Among the blocks of the later $MIF(k)$ are included the sets $\{\alpha\}\sqcup(B\smallsetminus\{\beta\})$, $\beta\in B\in\mathcal{F}$; hence the name. One might imagine that a method to reduce the number of points in a $MIF(k)$ can't have much to do with the problem of estimating the number $N(k)$ of the largest possible number of points in a $MIF(k)$. However, our final result (Theorem~\ref{N(k)}) is a new upper bound on $N(k)$ obtained by combining Theorem~\ref{NT(k,t)}  and Theorem~\ref{x_mapsto_y} with Bollob\'{a}s inequality~\eqref{Bollobas_inequality2}. Here we prove 
\begin{equation}\label{main_result}
N(k)\leq\frac{1}{2}\binom{2k-2}{k-1}+n(k,k-2).
\end{equation}
In view of Tuza's inequality \eqref{Tuza_upperbound_t}, this yields the bound
\begin{align}
N(k)\leq&\frac{3}{2}\overset{k-1}{\underset{i=1}\sum}\binom{2i}{i}-\frac{1}{2}\binom{2k-2}{k-1}\sim\frac{3}{2}\binom{2k-2}{k-1}.\label{asymptotic}
\end{align}
Thus as $k\rightarrow\infty$, $N(k)$ is at most $3$ times the lower bound \eqref{lower_bound} of Erd\H{o}s and Lov\'{a}sz.

In \cite{MR737092}, Hanson and Toft proved that, actually, $N(k)=2k-2+\frac{1}{2}\binom{2k-2}{k-1}$ for $2\leq k\leq 4$. In conjunction with Tuza's bound \eqref{asymptotic_Tuza} and its improvement \eqref{asymptotic}, this result leads us to pose:
\begin{conjecture}
For $k\geq 2$, $N(k)=2k-2+\frac{1}{2}\binom{2k-2}{k-1}$.
\end{conjecture}
It may be noted that Tuza constructed (\cite[Construction~11]{MR811117}) a $k-$uniform $1-$critical family with $2k-4+2\binom{2k-4}{k-2}$ points. This number is larger than $2k-2+\frac{1}{2}\binom{2k-2}{k-1}$ for $k\geq5$. However, as already noted, the class of $1-$critical uniform families is larger than that of $MIF$'s. Indeed, the families constructed by Tuza are not $MIF$'s. So this construction does not disprove the above conjecture.

Recall that the \emph{chromatic number} of a family is the smallest number of colours using which the points may be coloured so that no monochromatic block occurs. It is trivial to see that any uniform intersecting family (clique) $\mathcal{F}$ has chromatic number at most $3$. (Choose $x\in B\in\mathcal{F}$. Assign the first colour to $x$, second colour to the other points of $B$ and the third colour to the remaining points.) Thus any such family is either $2-$chromatic or $3-$chromatic. The article \cite{MR0382050} was mainly concerned with $k-$uniform $3-$chromatic intersecting families. This is a subclass of the class of $MIF(k)$'s. Indeed, a $k-$uniform intersecting family is $3-$chromatic if and only if its blocks are the only minimal (as opposed to just minimum sized) blocking sets. (The \emph{finite projective planes} of order $q\geq3$ are examples of $2-$chromatic $MIF(q+1)$.) So of course, the lower bound \eqref{lower_bound} holds for all $MIF(k)$'s.  

Finally, we note that in \cite{MR811117}, Tuza has made a precise conjecture  on the numbers $n(k,t)$:
\begin{conjecture}[Tuza]
For $k\geq t+2$, 
\begin{equation*}
n(k,t)=\left\lceil\frac{k}{t+1}\right\rceil\binom{\lfloor\frac{kt}{t+1}\rfloor+t}{t}+\left\lfloor\frac{kt}{t+1}\right\rfloor+t
\end{equation*}
\end{conjecture}
If this is correct, our bound \eqref{main_result} becomes
\begin{align*}
N(k)\leq& \frac{1}{2}\binom{2k-2}{k-1}+2\binom{2k-4}{k-2}+2k-4\sim\binom{2k-2}{k-1},
\end{align*}
which is asymptotically double the conjectured value.

\section{Proofs}

Recall that, for any finite family $\mathcal{F}$, $\vs{\mathcal{F}}$ is its number of points and $P_{\mathcal{F}}$ is its point set. If $\mathcal{F}$ is uniform, $\uk{\mathcal{F}}$ is its common block size. $\mathcal{F}^{\top}$ is the family of transversals of $\mathcal{F}$ and $\tr{\mathcal{F}}$ is the common size of the transversals. $N(k)$ is the maximum of $\vs{\mathcal{F}}$ over all $MIF(k)$ $\mathcal{F}$. $N^{\top}(k,t)$ is the maximum of $\vs{\mathcal{F}^{\top}}$ over all $\mathcal{F}$ with $\uk{\mathcal{F}}=k$ and $\tr{\mathcal{F}}=t$. Also $n(k,t)$ is the maximum of $\vs{\mathbb{I}}$ over all $ISP(k,t)$ $\mathbb{I}$.
 
\begin{theorem}\label{k^t}
If $\uk{\mathcal{F}}=k$ and $\tr{\mathcal{F}}=t$ then $\#(\mathcal{F}^{\top})\leq k^{t}$.
\end{theorem}
\begin{proof}
This is the $s=0$ case of the following.

\noindent{{\sf Claim} :} For $0\leq s\leq t$, any set of $s$ points of $\mathcal{F}$ are together contained in at most $k^{t-s}$ transversals of $\mathcal{F}$.

\vspace{2mm}

\noindent{\tt {Proof of the Claim} :} We prove this claim by backward induction on $s$. It is trivial for $s=t$. So suppose the claim holds for some $s$, with $1\leq s\leq t$. Take any set $A$ of $s-1$ points. Since $\tr{\mathcal{F}}=t>\#(A)$, $A$ is not a blocking set of $\mathcal{F}$. So there is a block $B\in\mathcal{F}$ disjoint from $A$. Therefore each transversal containing $A$ contains at least one of the $k$ sets $A\sqcup\{x\}$, $x\in B$. By induction hypothesis, $A\sqcup\{x\}$ is contained in at most $k^{t-s}$ transversals for each $x\in B$. Therefore $A$ is contained in at most $k.k^{t-s}=k^{t-(s-1)}$ transversals. This completes the induction.
\end{proof}

\begin{corollary}
Let $k,t$ be positive integers. Then up to isomorphism, there are only finitely many families $\mathcal{G}$ with $\uk{\mathcal{G}}=t$ such that $\mathcal{G}$ is isomorphic to $\mathcal{F}^{\top}$ for some uniform family $\mathcal{F}$ with $\uk{\mathcal{F}}=k$. 
\end{corollary}
\begin{proof}
By Theorem~\ref{k^t}, any such $\mathcal{G}$ has at most $k^{t}$ blocks; hence it has at most $t.k^{t}$ points. Therefore up to isomorphism, we may assume that all such families $\mathcal{G}$ are contained in the power set of a fixed set of size $tk^{t}$. So there are only finitely many $\mathcal{G}$'s. 
\end{proof}

This corollary shows that $N(k)$ and $N^{\top}(k,t)$ are both finite.

\begin{construction}\label{bg(k,t)}
Let $2\leq t\leq k-1$ and $S$ be a set of $k+t-2$ symbols. Let $\binom{S}{i}$ denote the family consisting of all $i-$subsets of $S$. Take a new symbol $x_{A}$ (from outside $S$) for each $A\in\binom{S}{k-1}$. Let
\begin{equation*}
\mathcal{F}=\binom{S}{k}\sqcup\left\{\{x_{A}\}\sqcup A: A\in\binom{S}{k-1}\right\}.
\end{equation*}
It is easy to verify that $\tr{\mathcal{F}}=t$ and
\begin{equation*}
\mathcal{F}^{\top}=\binom{S}{t}\sqcup\left\{\{x_{A}\}\sqcup(S\smallsetminus A):A\in\binom{S}{k-1}\right\}.
\end{equation*}
\end{construction}

\begin{theorem}\label{NT(k,t)}
For $2\leq t\leq k-1$, 
\begin{equation*}
k+t-2+\binom{k+t-2}{t-1}\leq N^{\top}(k,t)\leq n(k,t-1).
\end{equation*}
\end{theorem}
\begin{proof}
Construction~\ref{bg(k,t)} yields a $k-$uniform family $\mathcal{F}$  such that $\tr{\mathcal{F}}=t$ and $\mathcal{F}^{\top}$ has $k+t-2+\binom{k+t-2}{t-1}$ points. Hence we get the lower bound.

Let $\mathcal{F}$ be a $k-$uniform family with $\tr{\mathcal{F}}=t$. We need to show that $\vs{\mathcal{F}^{\top}}\leq n(k,t-1)$. Let $$\mathcal{E}=\{B_{i}:1\leq i\leq n\}$$ be a minimal subfamily of $\mathcal{F}$ such that $\tr{\mathcal{E}}=t$. Then, for $1\leq i\leq n$, $\mathcal{E}_{i}:=\mathcal{E}\smallsetminus\{B_{i}\}$ has $\tr{\mathcal{E}_{i}}=t-1$. Choose a transversal $T_{i}$ of $\mathcal{E}_{i}$, where $1\leq i\leq n$. Since $\tr{\mathcal{E}}=t$, it follows that $T_{i}\cap B_{i}=\emptyset$. Thus $$\mathbb{I}=\left\{(B_{i},T_{i}):1\leq i\leq n\right\}$$ is an $ISP(k,t-1)$. Therefore, to complete the proof, it suffices to show that each point $x$ of $\mathcal{F}^{\top}$ is a point of $\mathbb{I}$. Choose a transversal $T$ of $\mathcal{F}$ such that $x\in T$. Then $T$ intersects all the $B_{i}$'s. If $x$ was not a point of $\mathcal{E}$ then $T\smallsetminus\{x\}$ would be a blocking set of $\mathcal{E}$, of size $t-1$, contradicting the choice of $\mathcal{E}$. So $x$ is a point of $\mathcal{E}$ and 
hence of $\mathbb{I}$.
\end{proof}

Since, clearly, $N(k)\leq N^{\top}(k,k)$, Theorem~\ref{NT(k,t)} includes Tuza's upper bound \eqref{N_Tuza} on $N(k)$.

\begin{construction}\label{xmapstoy}
Let $\mathcal{F}$ be a $MIF(k)$ and suppose $\alpha\neq\beta$ are two points of $\mathcal{F}$ such that no block of $\mathcal{F}$ contains $\{\alpha,\beta\}$. Let $\mathcal{G}:=\{B\in\mathcal{F}:\alpha\notin B, \beta\notin B\}$. Put 
\begin{equation*}
\mathcal{F}[\beta\mapsto\alpha]:=\mathcal{G}\sqcup\{T\sqcup\{\alpha\}:T\in\mathcal{G}^{\top}\}.
\end{equation*}
\end{construction}

\begin{theorem}\label{x_mapsto_y}
Let $\alpha,\beta$ be two points of a $MIF(k)$ $\mathcal{F}$ such that no block of $\mathcal{F}$ contains both $\alpha$ and $\beta$. Then the family $\mathcal{F}[\beta\mapsto\alpha]$ (given by \emph{Construction~\ref{xmapstoy}}) is a $MIF(k)$ with point set $P_{\mathcal{F}}\smallsetminus\{\beta\}$.
\end{theorem}
\begin{proof}
Let $\mathcal{G}$ be as in Construction~\ref{xmapstoy}. If $T$ is transversal of $\mathcal{G}$ with $\#(T)\leq k-2$, then $T\sqcup\{\alpha,\beta\}$ is a blocking set of $\mathcal{F}$ of size at most $k$. Since $\mathcal{F}$ is a $MIF(k)$, it follows that $T\sqcup\{\alpha,\beta\}$ is a block of $\mathcal{F}$. This is a contradiction since no block of $\mathcal{F}$ contains both $\alpha$ and $\beta$. Thus $\tr{\mathcal{G}}\geq k-1$. Since, for $\beta\in B\in\mathcal{F}$, $B\smallsetminus\{\beta\}$ is a blocking set of $\mathcal{G}$, it follows that $\tr{\mathcal{G}}=k-1$. Thus $\widehat{\mathcal{F}}:=\mathcal{F}[\beta\mapsto\alpha]$ is uniform with $\uk{\widehat{\mathcal{F}}}=k$. This argument also shows that if $\beta\notin B\in\mathcal{F}$, then $B$ is a block of $\widehat{\mathcal{F}}$. Also if $\beta\in B\in\mathcal{F}$, then $\{\alpha\}\sqcup(B\smallsetminus\{\beta\})$ is a block of $\widehat{\mathcal{F}}$. We have the following.

\noindent\textsf{Claim  :} For each $T\in\mathcal{G}^{\top}$ there exists $T^{'}\in\mathcal{G}^{\top}$ such that $T\cap T^{'}=\emptyset$.
\begin{proof}[\tt {Proof of the Claim} :]\renewcommand{\qedsymbol}{}
Suppose the claim is false. Then there exists $T\in\mathcal{G}^{\top}$ such that $T$ is a blocking set of $\mathcal{G}^{\top}$. So $T$ is a blocking set of $\mathcal{G}\sqcup\mathcal{G}^{\top}$, and hence of $\mathcal{F}$. This means $\tr{\mathcal{F}}\leq \#(T)=k-1$. Contradiction.
\end{proof}
Let $C$ be a blocking set of $\widehat{\mathcal{F}}$. Then in particular it is a blocking set of $\mathcal{G}$. Since $\tr{\mathcal{G}}=k-1$ it follows that $\#(C)\geq k-1$. If $\#(C)=k-1$ then $C\in\mathcal{G}^{\top}$, so that $\alpha\notin C$. By the above claim there exists a $T\in\mathcal{G}^{\top}$ such that $T\cap C=\emptyset$. Hence $C$ is disjoint from $T\sqcup\{\alpha\}\in\widehat{\mathcal{F}}$. Contradiction. Hence $\#(C)\geq k$. Therefore $\tr{\widehat{\mathcal{F}}}=k$. Since $\mathcal{F}$ is an intersecting family, the construction of $\widehat{\mathcal{F}}$ shows that $\widehat{\mathcal{F}}$ is an intersecting family. Consequently $\widehat{\mathcal{F}}\subseteqq(\widehat{\mathcal{F}})^{\top}$. If $T$ is a transversal of $\widehat{\mathcal{F}}$ and $\alpha\in T$, then $T\smallsetminus\{\alpha\}$ is a transversal of $\mathcal{G}$, so that $T=(T\smallsetminus\{\alpha\})\sqcup\{\alpha\}\in\widehat{\mathcal{F}}$. If $T$ is a transversal of $\widehat{\mathcal{F}}$ and $\alpha\notin T$ then (as all 
the blocks of $\mathcal{F}$ with $\beta\notin B$ are blocks of $\widehat{\mathcal{F}}$ and for $\beta\in B\in\mathcal{F}$, $(B\smallsetminus\{\beta\})\sqcup\{\alpha\}$ is a block of $\widehat{\mathcal{F}}$) $T$ is a transversal of $\mathcal{F}$. Hence $T\in\mathcal{F}$ and $\beta,\alpha\notin T$, so that $T\in\mathcal{G}\subseteq\widehat{\mathcal{F}}$. Thus $(\widehat{\mathcal{F}})^{\top}\subseteqq\widehat{\mathcal{F}}$, so that $\widehat{\mathcal{F}}$ is a $MIF(k)$.

Clearly the point set of $\widehat{\mathcal{F}}$ is contained in $P_{\mathcal{F}}\smallsetminus\{\beta\}$. Take any $\gamma\in P_{\mathcal{F}}\smallsetminus\{\beta\}$. Take a block $B$ of $\mathcal{F}$ such that $\gamma\in B$. If $\beta\notin B$ then we have $\gamma\in B\in\widehat{\mathcal{F}}$ and hence $\gamma$ is a point of $\widehat{\mathcal{F}}$. If $\beta\in B$, then \--- as $\#(B)=k=\tr{\mathcal{F}}$, there is a block $B^{'}$ of $\mathcal{F}$ such that $B\cap B^{'}=\{\gamma\}$. Then $\gamma\in B^{'}\in\widehat{\mathcal{F}}$, hence again $\gamma$ is a point of $\widehat{\mathcal{F}}$. Thus the point set of $\widehat{\mathcal{F}}$ is $P_{\mathcal{F}}\smallsetminus\{\beta\}$.
\end{proof}

\begin{theorem}\label{N(k)}
For $k\geq 2$, 
\begin{equation*}
N(k)\leq\frac{1}{2}\binom{2k-2}{k-1}+n(k,k-2).
\end{equation*}
\end{theorem}
\begin{proof}
Let $\mathcal{F}$ be a $MIF(k)$. We need to show that 
$\vs{\mathcal{F}}\leq\frac{1}{2}\binom{2k-2}{k-1}+n(k,k-2)$.
Fix a point $\alpha$ of $\mathcal{F}$. We inductively define two finite sequences: a sequence $\{\beta_{n}:0\leq n\leq N-1\}$ of distinct points of $\mathcal{F}$ and a sequence $\{\mathcal{F}_{n}:1\leq n\leq N\}$ of $MIF(k)$'s. Define $\beta_{0}=\alpha$, $\mathcal{F}_{1}=\mathcal{F}$. Suppose we have already defined $\beta_{m}$ for $0\leq m\leq n-1$, and $\mathcal{F}_{m}$ for $1\leq m\leq n$. If for each point $\beta$ of $\mathcal{F}_{n}$ there is a block of $\mathcal{F}_{n}$ containing both $\alpha$ and $\beta$, then put $n=N$ and terminate the construction.  Otherwise, choose a point $\beta_{n}$ of $\mathcal{F}_{n}$ such that no block of $\mathcal{F}_{n}$ contains both $\alpha$ and $\beta_{n}$ and construct $\mathcal{F}_{n+1}:=\mathcal{F}_{n}[\beta_{n}\mapsto\alpha]$. By construction and Theorem~\ref{x_mapsto_y}, for $n\geq1$ each $\mathcal{F}_{n+1}$ is a $MIF(k)$ with $P_{\mathcal{F}_{n+1}}=P_{\mathcal{F}_{n}}\smallsetminus\{\beta_{n}\}$. 

Notice that this construction must end in finitely many steps, since by Theorem~\ref{k^t}, $\mathcal{F}_{1}=\mathcal{F}$ is finite. Since induction has terminated at the $N-$th step, $\mathcal{F}_{N}$ has the property that for each point $\beta$ of $\mathcal{F}_{N}$ there is a block of $\mathcal{F}_{N}$ containing both $\alpha$ and $\beta$. Put $$\mathcal{G}=\left\{B\in\mathcal{F}_{N}:\alpha\notin B\right\}.$$ 
For $\alpha\in B\in\mathcal{F}_{N}$, $B\smallsetminus\{\alpha\}$ is a blocking set of $\mathcal{G}$ of size $k-1$. So $\tr{\mathcal{G}}\leq k-1$. If $T$ is a transversal of $\mathcal{G}$ with $\#(T)\leq k-1$, then $T\sqcup\{\alpha\}$ is a blocking set of $\mathcal{F}_{N}$ with size at most $k$. Since $\mathcal{F}_{N}$ is a $MIF(k)$, it follows that $T\sqcup\{\alpha\}$ is a block of $\mathcal{F}_{N}$. Thus $\tr{\mathcal{G}}=k-1$ and $\mathcal{G}^{\top}=\left\{B\smallsetminus\{\alpha\}:\alpha\in B\in\mathcal{F}_{N}\right\}$.
Thus $P_{\mathcal{G}}=P_{\mathcal{G}^{\top}}=P_{\mathcal{F}}\smallsetminus\{\beta_{n}:0\leq n\leq N-1\}$. Therefore, by Theorem~\ref{NT(k,t)}, 
\begin{equation}\label{N(k)_1}
\vs{\mathcal{F}}=N+\vs{\mathcal{G}^{\top}}\leq N+N^{\top}(k,k-1)\leq N+n(k,k-2).
\end{equation}
Choose two blocks $B_{0}$, $B^{'}_{0}$ of $\mathcal{F}=\mathcal{F}_{1}$ such that $B_{0}\cap B^{'}_{0}=\{\beta_{0}\}$. Also, for $1\leq n\leq N-1$, choose two blocks $B_{n}$, $B^{'}_{n}$ of $\mathcal{F}_{n}$ such that $B_{n}\cap B^{'}_{n}=\{\beta_{n}\}$. (As already remarked, any point of a $MIF(k)$ lies in such a pair of blocks.) Put $T_{n}=B_{n}\smallsetminus\{\beta_{n}\}$, $T^{'}_{n}=B^{'}_{n}\smallsetminus\{\beta_{n}\}$. Thus $T_{n}\cap T^{'}_{n}=\emptyset$ for $0\leq n\leq N-1$. 

\noindent\textsf{Claim  :} For $0\leq m<n\leq N-1$, $T_{m}\sqcup\{\alpha\}$ and $T^{'}_{m}\sqcup\{\alpha\}$ are blocks of $\mathcal{F}_{n}$.
\begin{proof}[\tt {Proof of the Claim} :]\renewcommand{\qedsymbol}{}
This claim may be proved by finite induction on $n$.\\
If $n=m+1$, then $\mathcal{F}_{n}=\mathcal{F}_{m}[\beta_{m}\mapsto\alpha]$ and $T_{m}\sqcup\{\beta_{m}\}$, $T^{'}_{m}\sqcup\{\beta_{m}\}\in\mathcal{F}_{m}$ implies $T_{m}\sqcup\{\alpha\}$, $T^{'}_{m}\sqcup\{\alpha\}\in\mathcal{F}_{m+1}=\mathcal{F}_{n}$. If $m<n\leq N-1$, and the claim is correct for this value of $n$, then $T_{m}\sqcup\{\alpha\}$, $T^{'}_{m}\sqcup\{\alpha\}\in\mathcal{F}_{n}$ and $\mathcal{F}_{n+1}=\mathcal{F}_{n}[\beta_{n}\mapsto\alpha]$ implies $T_{m}\sqcup\{\alpha\}$, $T^{'}_{m}\sqcup\{\alpha\}\in\mathcal{F}_{n+1}$.
\end{proof}

Now for $0\leq m<n\leq N-1$, $T_{m}\sqcup\{\alpha\}$, $T^{'}_{m}\sqcup\{\alpha\}$, $T_{n}\sqcup\{\beta_{n}\}$ and $T^{'}_{n}\sqcup\{\beta_{n}\}$ are blocks of the intersecting family $\mathcal{F}_{n}$. Therefore these four sets intersect pairwise. Since $\beta_{n}\neq\alpha$ and $\alpha$ \& $\beta_{n}$ are never together in a block of $\mathcal{F}_{n}$, it follows that $T_{m}\cap T_{n}\neq\emptyset$, $T^{'}_{m}\cap T_{n}\neq\emptyset$, $T_{m}\cap T^{'}_{n}\neq\emptyset$ and $T^{'}_{m}\cap T^{'}_{n}\neq\emptyset$ for $0\leq m<n\leq N-1$. Therefore,
\begin{equation*}
\mathbb{I}:=\left\{(T_{n},T^{'}_{n}):0\leq n\leq N-1\right\}\sqcup\left\{(T^{'}_{n},T_{n}):0\leq n\leq N-1\right\}
\end{equation*}
is an $ISP(k-1,k-1)$ containing $2N$ pairs. Therefore by Bollob\'{a}s inequality~\eqref{Bollobas_inequality2}, we get 
\begin{equation}\label{N(k)_2}
N\leq\frac{1}{2}\binom{2k-2}{k-1}.
\end{equation}
From \eqref{N(k)_1} and \eqref{N(k)_2}, we conclude that $\vs{\mathcal{F}}\leq\frac{1}{2}\binom{2k-2}{k-1}+n(k,k-2)$.
\end{proof}

\begin{acknowledgement}
We thank Professor Bhaskar Bagchi for suggesting the problem and for his help in the preparation of this paper. 
 
\end{acknowledgement}

\end{document}